\documentclass{amsart}
\usepackage{graphicx} 
\usepackage{epstopdf}
\usepackage[mathscr]{eucal}
\theoremstyle{plain}
\newtheorem{prop}{Proposition}[section]
\newtheorem{coro}[prop]{Corollary}

\newtheorem{conj}[prop]{Conjecture}
\newtheorem{lemm}[prop]{Lemma}

\newtheorem{thm}[prop]{Theorem}
\newtheorem{ex}[prop]{Example}
\theoremstyle{definition}
\newtheorem{defn}[prop]{Definition}

\newtheorem{rem}[prop]{Remark}

\DeclareMathOperator{\sign}{sign}

\def\mcg#1;#2{\Gamma_{#1,#2}}
\def\fg#1;#2{\Pi_{#1,#2}}
\def\tb#1;#2{\mathscr{K}_{\frac{#1}{#2}}}

\begin{document}

\title[Characterization of Quasi-alternating Montesinos Links]
{Characterization of Quasi-alternating Montesinos Links}

\keywords{quasi-alternating links, determinant, Montesinos links}

\author{Khaled Qazaqzeh}
\address{Department of Mathematics\\ Yarmouk University\\
Irbid, Jordan, 21163}
\email{qazaqzeh@yu.edu.jo}
\urladdr{http://faculty.yu.edu.jo/qazaqzeh}

\author{Nafaa Chbili}
\address{Department of Mathematical Sciences\\ College of Science UAE University \\ 17551 Al Ain, U.A.E.}
\email{nafaachbili@uaeu.ac.ae}
\urladdr{http://faculty.uaeu.ac.ae/nafaachbili}

\author{Balkees Qublan}
\address{Department of Mathematics\\ Yarmouk University\\
Irbid, Jordan, 21163}
\email{balkesqublan@yahoo.com}

\thanks{The first and the third authors are supported by a grant from Yarmouk University.}
\date{19/05/2012}

\begin{abstract}
We construct an infinite family of quasi-alternating links from a given quasi-alternating link by replacing a crossing by a product of rational tangles each of which extends that crossing. Consequently, we determine an infinite family of quasi-alternating Montesinos links. This family contains all the classes of quasi-alternating Montesinos links that have been detected by Widmar in \cite{W}. We conjecture that this family contains all quasi-alternating Montesinos links up to mirror image that are not alternating and this will characterize all quasi-alternating Montesinos links.
\end{abstract}

\maketitle

\section{introduction}

Quasi-alternating links were introduced by Ozsvath and Szabo in \cite[Definition.\,3.9]{OS} as a natural generalization of alternating links. The definition is given in a recursive way:

\begin{defn}
The set $\mathcal{Q}$ of quasi-alternating links is the smallest set
satisfying the following properties:
\begin{itemize}
	\item The unknot belongs to $\mathcal{Q}$.
  \item If $L$ is link with a diagram containing a crossing $c$ such that
\begin{enumerate}
\item both smoothings of the diagram of $L$ at the crossing $c$, $L_{0}$ and $L_{1}$ as in figure \ref{figure} belong to $\mathcal{Q}$, 
\item $\det(L_{0}), \det(L_{1}) \geq 1$,
\item $\det(L) = \det(L_{0}) + \det(L_{1})$; then $L$ is in $\mathcal{Q}$ and in this case we say $L$ is quasi-alternating at the crossing $c$.
\end{enumerate}
\end{itemize}
\end{defn}

\begin{figure} [h]
\begin{center}
\includegraphics[width=10cm,height=2cm]{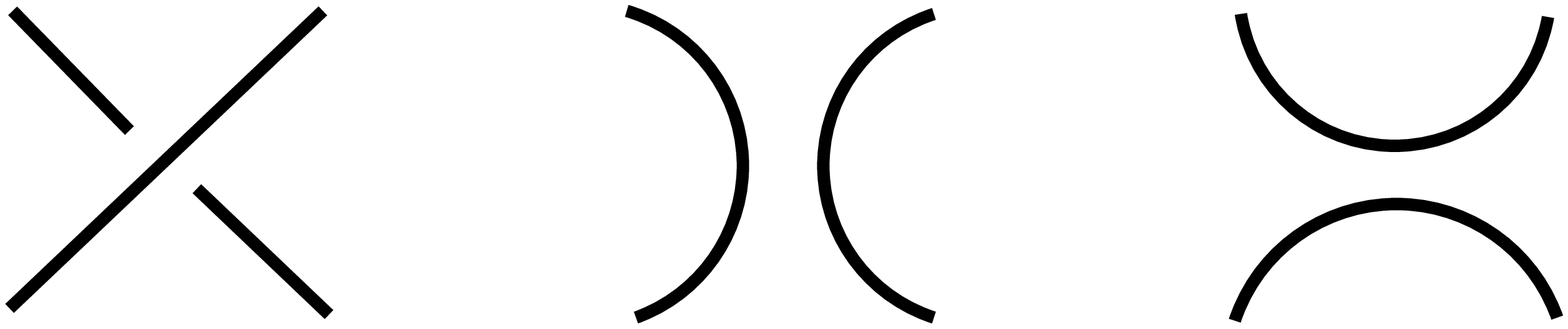} \\
{$L$}\hspace{3.5cm}{$L_0$}\hspace{3.5cm}$L_{1}$
\end{center}
 %\vspace{-3.5cm}
\caption{The link diagram $L$ at crossing $c$ with $\sign(c) = -1$ and its smoothings $L_{0}$ and $L_{1}$ respectively.}\label{figure}
\end{figure}

Studying quasi-alternating links becomes an interesting problem in the last few years. It has been shown in \cite{MO} that they are homologically thin for both Khovanov homology and knot Floer homology. In other words, their Khovanov homology and knot Floer homology depend only on the Jones and the Alexander polynomials respectively and the signature of the given link.

The class of quasi-alternating links contains the class of alternating links as it has been shown in \cite[Lemma.\,3.2]{OS}. Champanerkar and Kofman in \cite{CK} determine an infinite family of quasi-alternating pretzel links by twisting any known quasi-alternating pretzel link. Later on, Greene in \cite{G} characterized all quasi-alternating pretzel links. In this paper, we generalize the technique of twisting to construct an infinite family of quasi-alternating Montesinos links. This infinite family include all classes of quasi-alternating Montesinos links detected by Widmar in \cite{W}. Moreover, we conjecture that this family contains all quasi-alternating Montesinos links up to mirror image that are not alternating.

\section{Montesinos Links}

%The relevant notation and facts concerning Montesinos links follows \cite[Section.\,3.2]{OSS}. 
Montesinos links are natural generalization of rational links. They have been classified first by the authors of \cite{BoSi} and for further details and more recent reference see \cite{BZ}.

\begin{defn}
A Montesinos link has a diagram as shown in figure \ref{figure1}. In this diagram, $e$ is an integer that represents number of half twists and the box \fbox{$\alpha_{i}, \beta_{i}$} stands for a rational tangle with slope $\frac{\beta_{i}}{\alpha_{i}}$ of two coprime integers $\alpha_{i} > 1$ and $\beta_{i}$ whose continued fraction is given by
\[
\frac{\beta_{i}}{\alpha_{i}} = \cfrac{1}{a_{n}+\cfrac{1}{a_{n-1}+\cfrac{1}{\ddots +\cfrac{1}{a_{1}}}}}.
\]
The continued fraction for any rational number $\frac{\beta}{\alpha}$ will be denoted by $[a_{1},a_{2},\ldots,a_{n}]$ from now on.
\end{defn}

\begin{figure}[htbp]
	\begin{center}
  \includegraphics[width=6.5cm, height=6cm]{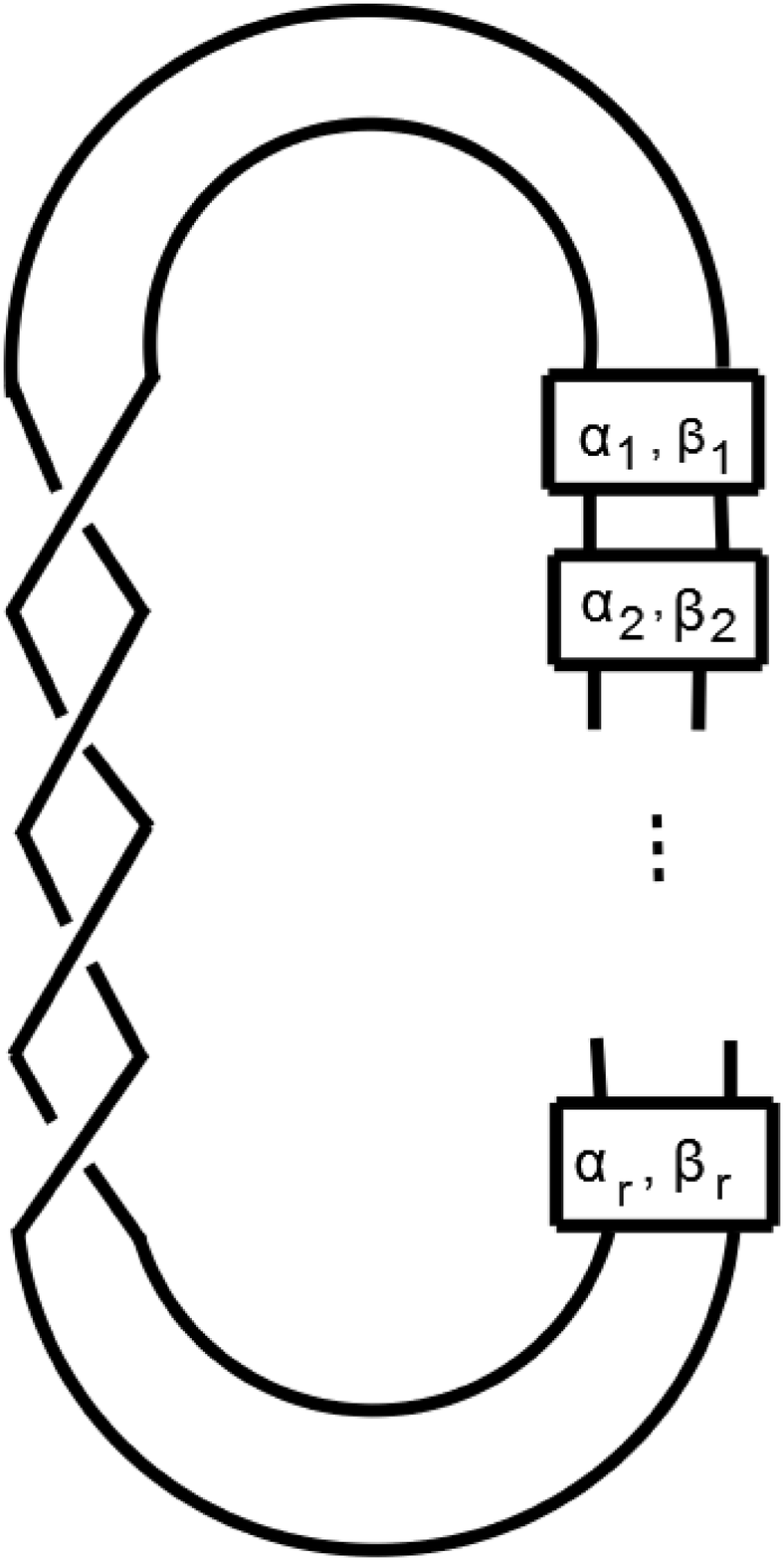} \hspace{2cm}
  \includegraphics[width=4.5cm, height=3.5cm]{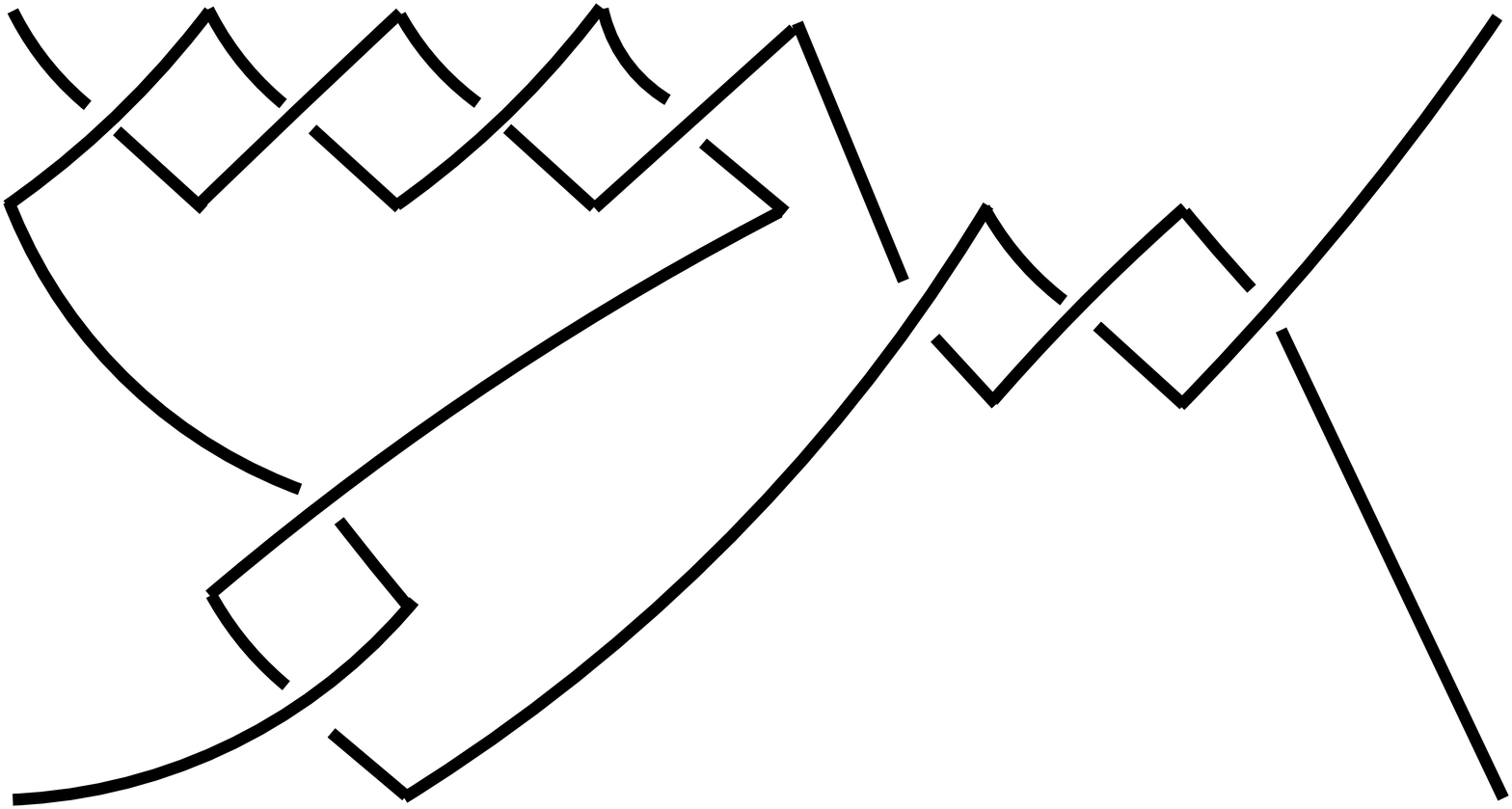}\\
(a)\hspace{6cm} (b)
\end{center}
\caption{Montesinos link diagram in (a) with $e = 4$. (b)  represents the rational tangle with slope $\frac{7}{31}$ with  continued fraction $[4,2,3]$.}\label{figure1}
\end{figure}

The Montesinos links are classified by the following theorem:
\begin{thm}\cite{B,BoSi}
The Montesions link $M(e;(\alpha_{1},\beta_{1}), (\alpha_{2},\beta_{2}), \dots, (\alpha_{r},\beta_{r}))$ with $r \geq 3$ and $\sum_{j = 1}^{r} \frac{1}{\alpha_{j}} \leq r -2$, is classified by the ordered set of fractions $(\frac{\beta_{1}}{\alpha_{1}}\pmod 1,\frac{\beta_{2}}{\alpha_{2}}\pmod 1, \ldots,\\ \frac{\beta_{r}}{\alpha_{r}}\pmod 1)$ up to cyclic permutations and reversal of order, together with the rational number $e_{0} = e - \sum_{j = 1}^{r} \frac{\beta_{j}}{\alpha_{j}}$.
\end{thm}

\begin{defn}
A Montesinos link $M(e;(\alpha_{1},\beta_{1}), (\alpha_{2},\beta_{2}), \dots, (\alpha_{r},\beta_{r}))$ is in standard form if $0 < \frac{\beta_{i}}{\alpha_{i}} < 1$ for $ i = 1,2, \ldots, r$.
\end{defn}
\begin{rem}
Any Montesinos link can be put into a standard form using the above theorem.
\end{rem}
\begin{defn}
A rational tangle extends the crossing $c$ if it contains the crossing $c$ and for all $i$ we have $\sign(c)a_{i} \geq 1$ (\cite[Page.\,2452]{CK}). The product of rational tangles extends the crossing $c$ if each rational tangle in the product extends the crossing $c$. For the definition of product of two rational tangles and more details of this subject see \cite[Section.\,2]{KL1}.
\end{defn}

For a given sequence of nonnegative integers $[a_{1},a_{2},\ldots,a_{n}]$, the authors of \cite{KL} define a new sequence of positive integers recursively by
\[
T(0) = 1, T(1) = a_{1},
T(m) = a_{m}T(m-1) + T(m-2).
\]

They show that $T(n)$ is equal to the determinant of the rational link $K_{\alpha,\beta}$ with slope $\frac{\beta}{\alpha}$ of two coprime integers $ \alpha > 1$ and $\beta$ whose continued fraction is $[a_{1},a_{2},\ldots,a_{n}]$. For our use, we define $\overleftarrow{T}(m)$ to be $T(m)$ after replacing $a_{i}$ with $a_{j}$ where $i + j = n + 1$. Therefore, we obtain
\[
\overleftarrow{T}(0) = 1, \overleftarrow{T}(1) = a_{n}, \overleftarrow{T}(m) = a_{n - m + 1}\overleftarrow{T}(m-1) + \overleftarrow{T}(m-2),
\] for the above sequence of nonnegative integers.

\begin{rem}
It is clear that if $[a_{1},a_{2},\ldots,a_{n}]$ is a continued fraction of $\frac{\beta}{\alpha}$, then $[-a_{1},-a_{2},\ldots,\\ -a_{n}]$ is a continued fraction of $\frac{-\beta}{\alpha}$.
\end{rem}
\begin{lemm}
For any positive rational number $\frac{\beta}{\alpha}$, there is a continued fraction $[a_{1},a_{2},\ldots,a_{n}]$ of nonnegative integers.
\end{lemm}
\begin{lemm}
Let $[a_{1},a_{2},\ldots,a_{n}]$ be a continued fraction of $\frac{\beta}{\alpha}$ of nonnegative integers, then   $T(n - 1) = \beta$ and $T(n) = \alpha$.
\end{lemm}
\begin{proof}
 We first show  that $T(n)$ and $T(n-1)$ are relatively prime. Using the recursive formula for $T(n)$, we can see that if  $d|T(n)$ and $d|T(n-1)$, then $d|T(n-2)$. Now if repeat this process we obtain that $d|T(0)=1$.\\
Now, we have
\begin{align*}
\frac{T(n)}{T(n-1)} & = \frac{a_{n}T(n-1) + T(n-2)}{T(n-1)} \\
& = a_{n} + \frac{T(n-2)}{T(n-1)} \\
& = a_{n} + \frac{1}{\frac{T(n-1)}{T(n-2)}}\\
& = a_{n} + \frac{1}{a_{n-1} + {\frac{T(n-3)}{T(n-2)}}}\\
&  \ \ \vdots  \\
& = a_{n} + \cfrac{1}{a_{n-1}+\cfrac{1}{\ddots +
\cfrac{1}{a_{1} }}} = \frac{\alpha}{\beta}.
\end{align*}

\end{proof}
\begin{lemm}
For any sequence of nonnegative integers $[a_{1},a_{2}, \ldots, a_{n}]$ with $a_{m} > 1$, we have
\[
T(m-1) + T_{1}(m)  = T(m),
\]
\[
\overleftarrow{T}(m-1) + \overleftarrow{T}_{1}(m)  = \overleftarrow{T}(m),
\]
where $T_{1}(m)$, and $ \overleftarrow{T}_{1}(m)$ are $T(m)$, and $ \overleftarrow{T}(m)$ for the sequence $[a_{1},a_{2}, \ldots, a_{m}-1,\ldots ,a_{n}]$ respectively.
\end{lemm}

\begin{proof}
The result follows since
\[
T_{1}(m) = (a_{m}-1)T(m-1) + T(m-2),
\]
and
\[
\overleftarrow{T}_{1}(m) = (a_{n-m+1} -1)\overleftarrow{T}(m-1) + \overleftarrow{T}(m-2).
\]
\end{proof}

\section{Main Results}
We first prove  the following lemma that will be used in proof of the main result of this paper.
\begin{lemm}\label{new}
Let $L$ be a quasi-alternating link diagram at the crossing $c$, then
\[ \det(L') =
  \begin{cases}
  T(n)\det(L_{0}) + \overleftarrow{T}(n-1) \det(L_{1}),
     & \text{if $\sign(c) = -1 $},
     \\
    {T(n)\det(L_{1}) + \overleftarrow{T}(n - 1) \det(L_{0})}, & \text{if $\sign(c) = 1 $},
  \end{cases}
\]
where $L'$ is the link diagram obtained from $L$ by replacing the crossing $c$ by a rational tangle that extends $c$ with slope $\frac{T(n-1)}{T(n)}$.
\end{lemm}
\begin{proof}
We prove the case where $\sign(c) = -1$ and the other case will follow by taking mirror image. Let $[a_{1},a_{2},\ldots,a_{n}]$ be the continued fraction of nonnegative integers of the rational tangle with slope $\frac{T(n-1)}{T(n)}$.

We use induction on the number of denominators of the continued fraction. It is clear that the case holds for $n = 1$ by just doing simple induction on the length of $a_{1}$. We assume that the result holds for all rational tangles of number of denominators equal to $n -1$, i.e for the any rational tangle with continued fraction $[a_{1},a_{2},\ldots,a_{n-1}]$.
Now we want to show that the result holds for any rational tangle of continued fraction $[a_{1},a_{2},\ldots,a_{n}]$. Again we apply induction but this time on $a_{n}$. In the case that $a_{n} = 1$, then the continued fraction is given by the sequence $[a_{1},a_{2},\ldots,a_{n-1}+1]$ and hence the result follows by the induction hypothesis. Now suppose that the result holds for $a_{n} = k$, we want to show the result for $a_{n} = k + 1$. We have to consider two cases of $n$ being either even or odd. First, we assume that $n$ is odd. Now we smooth one of the crossings in $a_{n}$ and we obtain
\begin{align*}
\det(L') & = \det(L'_{0}) + \det(L'_{1})\\& = T(n-1)\det(L_{0}) + \overleftarrow{T}(n-2)\det(L_{1}) + T_{1}(n)\det(L_{0}) + \overleftarrow{T}_{1}(n-1)\det(L_{1}) \\ & = (T(n-1) + T_{1}(n))\det(L_{0}) + (\overleftarrow{T}(n-2) + \overleftarrow{T}_{1}(n - 1))\det(L_{1}) \\ & = T(n)\det(L_{0}) + \overleftarrow{T}(n-1)\det(L_{1}).
\end{align*}

In the case that $n$ is even the result follows since the first two terms in the second equality equal to $\det(L_{1}')$ and the last two terms equal to $\det(L_{0}')$.
\begin{figure}[htbp]
\begin{center}
\includegraphics[width=5.5cm, height=4cm]{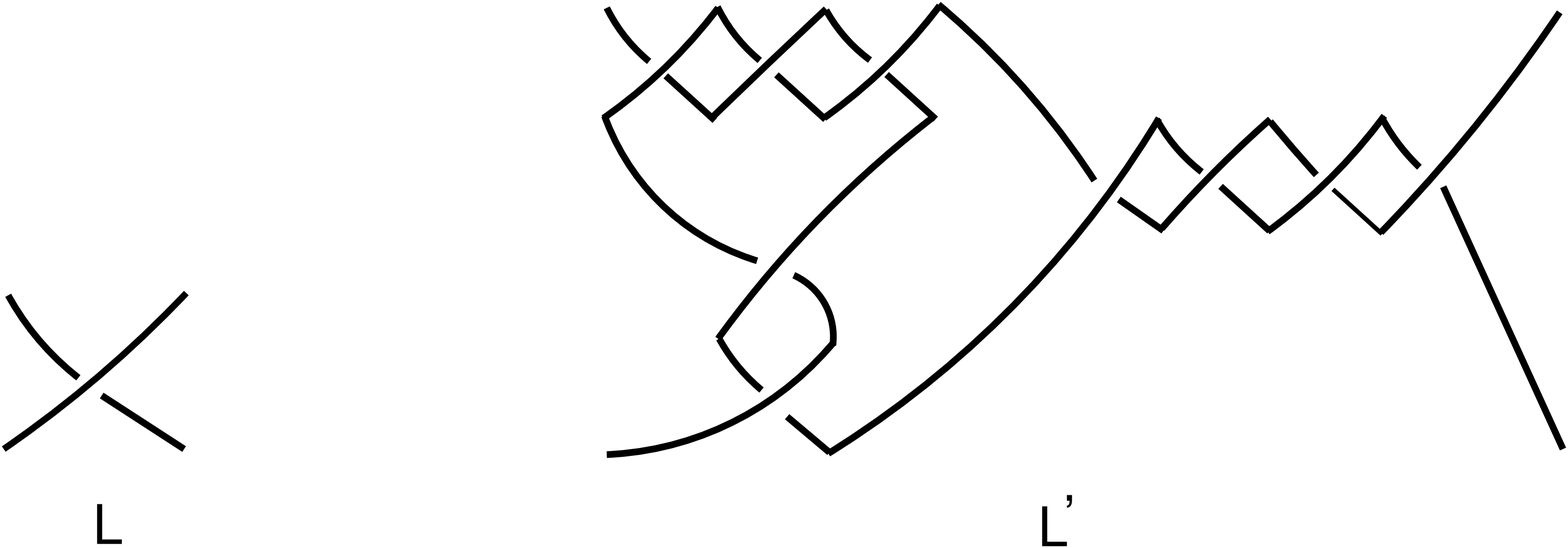} \hspace{1cm} 
\includegraphics[width=4.5cm, height=4cm]{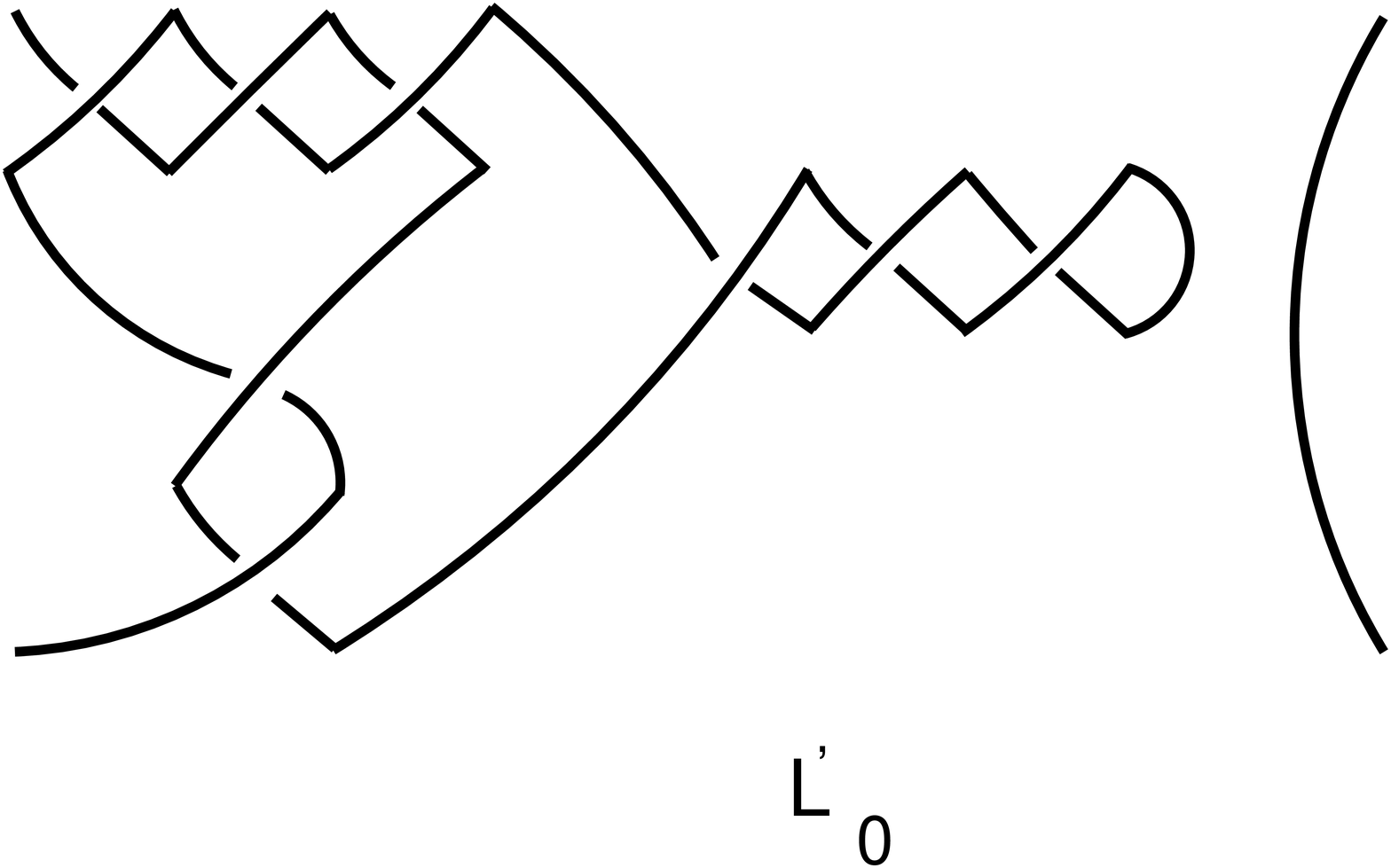} \hspace{1cm} 
\includegraphics[width=4.5cm, height=4cm]{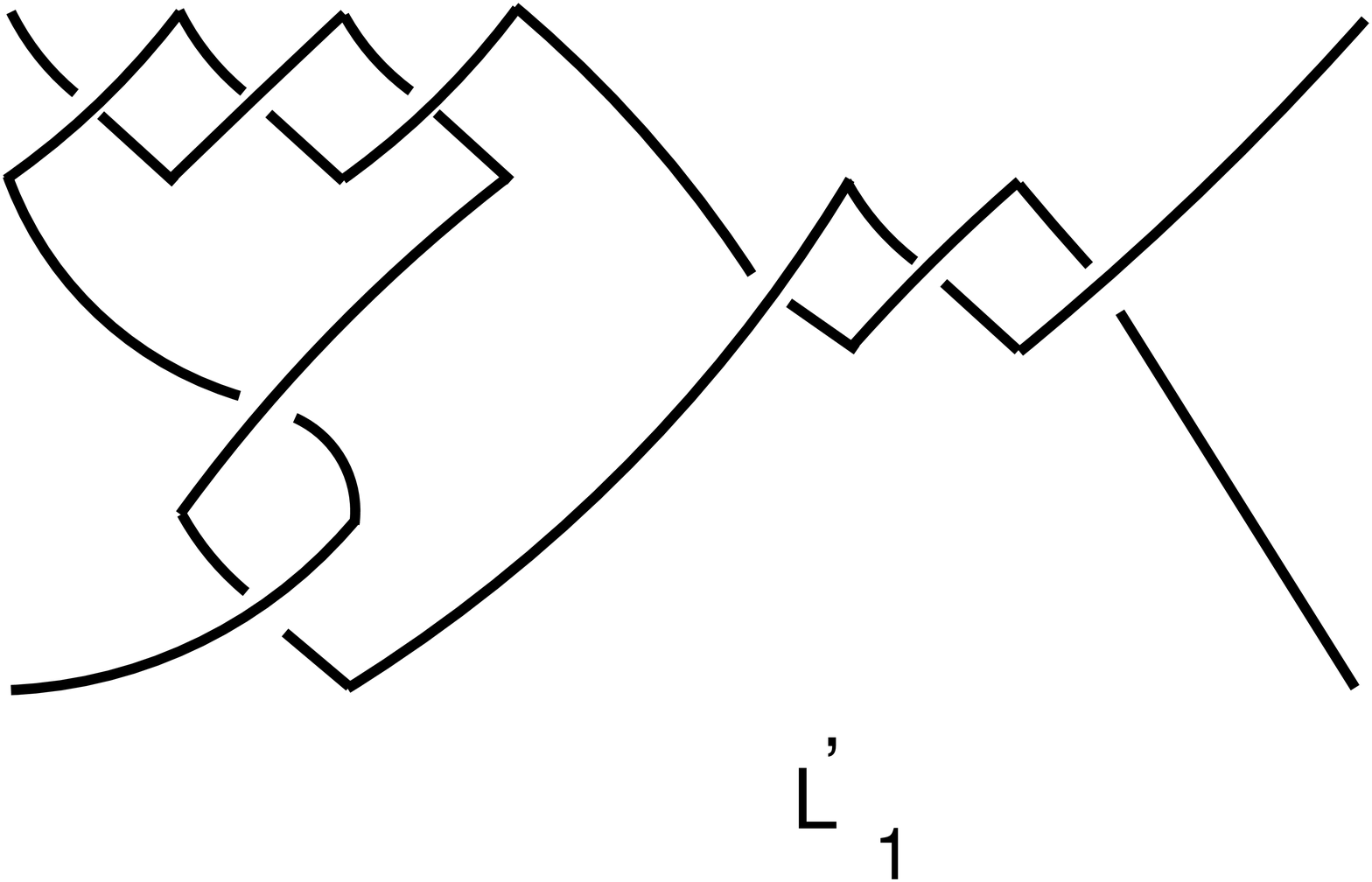}
\end{center}
\end{figure}
\end{proof}

Now we state our main theorem that is a natural generalization of \cite[Theorem.\,2.1]{CK}.

\begin{thm}\label{new2}
If $L$ is a quasi-alternating link diagram at some crossing $c$, then $L^{*}$ obtained by replacing the crossing $c$ by a product of rational tangles that extends the crossing $c$ is quasi-alternating.
\end{thm}

\begin{proof}
Let $L$ be a quasi-alternating link diagram at the crossing $c$. We may assume that $\sign(c) > 0$ as shown in figure \ref{figure} by taking the mirror image if it is needed. We follow the same notation as in the proof of \cite[Theorem.\,2.1]{CK}. To be more precise, we let $L^{n}$ denote the link diagram with $|n|$ additional crossings inserted at $c$, which are vertical half-twists if $n > 0$ and horizontal half-twists if $n < 0$.

To show our result, we know that $L^{2}$ is quasi-alternating as a result of \cite[Theorem.\,2.1]{CK}. Moreover, $L^{2}$ is quasi-alternating at any one of these two crossings $c_{1}, c_{2}$ that replaces $c$. Now we replace the crossing $c_{1}$ in $L^{2}$ by any rational tangle that extends $c_{1}$ to obtain a new quasi-alternating link diagram $L^{2'}$ at any crossing of this rational tangle as a result of \cite[Theorem.\,2.1]{CK}.

We show that $L^{2'}$ is quasi-alternating at the crossing $c_{2}$. Now if we smooth $L^{2'}$ at the crossing $c_{2}$, we obtain $L^{2'}_{0}$ that is the link diagram of $L'$ where $L'$ is the link diagram obtained from $L$ by replacing $c$ by the same rational tangle that has been used before to extend the crossing $c_{1}$ which quasi-alternating as a result of \cite[Theorem.\,2.1]{CK}. Also, we obtain the link diagram $L^{2'}_{1}$ that is a connected sum of a rational link and $L_{1}$ which is quasi-alternating as a result of \cite[Lemma.\,2.3]{CK}.

Finally, we check that the determinant is additive at the crossing $c_{2}$. We use Lemma \ref{new}, to obtain
\begin{align*}
\det(L^{2'}) = &  \ T(n)\det(L) + \overleftarrow{T}(n-1) \det(L_{1}) \\
& = T(n)(\det(L_{0}) + \det(L_{1})) + \overleftarrow{T}(n-1)\det(L_{1}) \\
& = T(n)\det(L_{0}) + \overleftarrow{T}(n-1)\det(L_{1}) + T(n)\det(L_{1})\\
& = \det(L') + \det(L^{2'}_{1}) \\
& = \det(L^{2'}_{0}) + \det(L^{2'}_{1}).
\end{align*}
\end{proof}

\begin{coro}\cite[Theorem.\,3.2(1)]{CK}
\begin{enumerate}
\item If $q > \min\{p_{1}, p_{2}, \ldots, p_{n}\}$, then the pretzel link $P(p_{1},p_{2},\ldots,p_{n},-q)$ is quasi-alternating for $n, p_{1}, p_{2}, \ldots, p_{n}, q \geq 2$.
\item If $p > \min\{q_{1}, q_{2}, \ldots, q_{m}\}$, then the pretzel link $P(p,-q_{1}, -q_{2}, \ldots, -q_{m})$ is quasi-alternating for $m, p, q_{1}, q_{2} \ldots, q_{m} \geq 2$.

\end{enumerate}

\end{coro}
\begin{proof}
We show the first case and the second case follows by taking the mirror image.
We may assume that $\min \{p_{1},p_{2}, \ldots, p_{n} \} =  p_{1}$. We start with the pretzel link diagram $P(p_{1},1,-q)$ to be our link $L$. It is clear that if we smooth $L$ at the crossing of the middle tassel then we obtain that $L_{0} = T(2,p_{1})\#T(2,-q)$ which is quasi-alternating by \cite[lemma.\,2.3]{CK} and $L_{1} = T(2, p_{1} -q)$ which is also quasi-alternating since it is alternating. Now we have
\[
\det(L) = p_{1}q + q - p_{1} = \det(L_{0}) + \det(L_{1}).
\]

Finally, we apply the theorem \ref{new2} to obtain the required result.
\end{proof}
\begin{thm}\label{main}
The Montesinos link $ L = M(e;(\alpha_{1},\beta_{1}),(\alpha_{2},\beta_{2}), \ldots,(\alpha_{r},\beta_{r}) )$ in standard form is quasi-alternating if one of the following conditions is satisfied
\begin{enumerate}
\item $ e \leq 0$;
\item $e \geq r + 1$;
\item $e = 1$  with $\frac{\alpha_{i}}{\alpha_{i} - \beta_{i}} > \min \{ \frac{\alpha_{j}}{\beta_{j}}\ | \  j \neq i \}$ for some $1 \leq i \leq r$;
\item $e = r$  with $\frac{\alpha_{i}}{\beta_{i}} > \min \{ \frac{\alpha_{j}}{\alpha_{j} - \beta_{j}}\ | \  j \neq i \}$ for some $1 \leq i \leq r$.
\end{enumerate}
\end{thm}

\begin{proof}
We prove each case separate
\begin{enumerate}
\item The Montesinos link with the given condition has an equivalent description $M(e;(\alpha_{1},\beta_{1}),\\ (\alpha_{2},\beta_{2}), \ldots,(\alpha_{r},\beta_{r}))$ in standard form with $ e \leq 0$. The associated diagram is connected and alternating, so the link is quasi-alternating.

\item It is enough to show that its mirror image is quasi-alternating. The mirror image of the given link has a description $M(-e;(\alpha_{1},-\beta_{1}), (\alpha_{2},-\beta_{2}), \ldots,(\alpha_{r},-\beta_{r}))$. Now this link has an equivalent description in standard form $M(-e + r;(\alpha_{1},\alpha_{1} - \beta_{1}), (\alpha_{2},\alpha_{2} - \beta_{2}), \ldots, \\ (\alpha_{r},\alpha_{r} - \beta_{r}))$. Hence the result follows from the first case since $-e + r \leq 0$.
\item In this case, we start with the Montesinos link diagram $L'$ with a description $M(0;(\alpha_{1},\beta_{1}),\\ (1,1), (\alpha_{i}, \beta_{i} - \alpha_{i}))$ with $\frac{\alpha_{i}}{\alpha_{i} - \beta_{i}} > \frac{\alpha_{1}}{\beta_{1}}$. 

We show that this diagram is quasi-alternating at the only crossing in the second integer tangle. It is easy to see that $L'_{0}$ is a 4-plate and by \cite{BS} we know that it is a rational link hence it is quasi-alternating. Also, $L'_{1}$ is a connected sum of two rational links hence it is quasi-alternating by \cite[Lemma.\,2.3]{CK}. To show the additivity of the determinant at this crossing, we use the main result of \cite{QQ} to obtain the determinant of the links $L'$ and $L'_{0}$
\[
\det(L') = \alpha_{1}\alpha_{i} + \alpha_{i}\beta_{1} - \alpha_{1}(\alpha_{i} - \beta_{i});
\]
\[
\det(L'_{0}) =|\alpha_{i}\beta_{1} - \alpha_{1}(\alpha_{i} - \beta_{i})|;
\]

Also, we know that
\[
\det(L'_{1})= \det(K_{\alpha_{1}, \beta_{1}}\#K_{\alpha_{i}, \beta_{i} - \alpha_{i}}) = \det(K_{\alpha_{1}, \beta_{1}})\det(K_{\alpha_{i}, \beta_{i} - \alpha_{i}}) = \alpha_{1}\alpha_{i}.
\]

For the determinant to be additive at that crossing, we have to have
\begin{align*}
\alpha_{i}\beta_{1} - \alpha_{1}(\alpha_{i} - \beta_{i}) & >  0 \\
\alpha_{i}\beta_{1} &  > \alpha_{1}(\alpha_{i} - \beta_{i})  \\
\frac{\alpha_{i}}{\alpha_{i} - \beta_{i}} & >  \frac{\alpha_{1}}{\beta_{1}}.
\end{align*}
Now since $L'$ is quasi-alternating at the given crossing, then we apply theorem \ref{new2} to obtain a quasi-alternating Montesinos link with the given description.
\item Again it is enough to obtain the result for the mirror image of the given link. The mirror image of the given link has a description $M(-r  ;(\alpha_{1},-\beta_{1}), (\alpha_{2},-\beta_{2}), \ldots,(\alpha_{r},-\beta_{r}))$. Now this link has an equivalent description in standard form $M(0;(\alpha_{1},\alpha_{1} - \beta_{1}), (\alpha_{2},\alpha_{2} - \beta_{2}), \ldots, (\alpha_{r},\alpha_{r} -\beta_{r}))$. Finally, the result follows from the third case.
\end{enumerate}

\end{proof}

\begin{prop}
The above theorem generalizes the sufficient condition of \cite[Theorem.\,1.4]{G}.
\end{prop}
\begin{proof}
All cases of \cite[Theorem.\,1.4]{G} follow directly from the cases of the above theorem by putting the pretzel link into a standard form except the second case. This case follows from the second case of the above theorem since the link has a description $M(-m+1;(p_{1},1),(p_{2},1),\ldots,(p_{n},1),(q_{1},-1),(q_{2},-1),\ldots, (q_{m},-1))$ that has an equivalent description $M(1;(p_{1},1),(p_{2},1),\ldots,(p_{n},1),(q_{1},q_{1}-1),(q_{2},q_{2}-1),\ldots, (q_{m},q_{m}-1))$ with $q_{1} > \min \{p_{1},p_{2}, \ldots, \\ p_{n},\frac{q_{2}}{q_{2}-1},\ldots,\frac{q_{m}}{q_{m}-1} \}$ since $q_{j} \geq 3$ for $1\leq j \leq m$.
\end{proof}

\begin{prop}
All classes of quasi-alternating Montesinos links detected by Widmar in \cite[Theorem.\,1.2 and the only Remark]{W} fall into the third case of
 theorem \ref{main}.
\end{prop}

\begin{proof}
We prove each case separate, but first we assume $R$ is any positive rational tangle of slope $\frac{\beta}{\alpha}$.
\begin{enumerate}
\item The link $L = L(a_{1}a_{2}, R, -n)$ with $ 1 + a_{1}(a_{2} - n) < 0 $ has an equivalent description in standard form $M(1;(a_{1}a_{2}+1, a_{1}), (\alpha,\beta),(n, n-1))$. So we have
\begin{align*}
1 + a_{1}(a_{2} - n) < 0  & \iff a_{1}(a_{2} - n) < -1 \\
& \iff a_{1}a_{2} - na_{1} <  -1 \\
& \iff -na_{1} < -(1 +  a_{1}a_{2}) \\
& \Rightarrow  n > \frac{a_{1}a_{2} + 1}{a_{1}} \geq \min \{\frac{a_{1}a_{2} + 1}{a_{1}}, \frac{\alpha}{\beta} \};
\end{align*}

\item The link $L = L(a_{1}a_{2}, R, (-c_{1})(-c_{2}))$ with $a_{2} \leq c_{2}$ and $a_{1} > c_{1}$ has an equivalent description in standard form $M(1;(a_{1}a_{2}+1, a_{1}), (\alpha,\beta),(c_{1}c_{2}+1,c_{1}c_{2} + 1 - c_{1})$. So we have
\begin{align*}
a_{2} \leq c_{2} \ \text{and} \ a_{1} > c_{1} & \Rightarrow a_{1}c_{1}c_{2} + a_{1} > c_{1}a_{1}a_{2} + c_{1} \\ & \Rightarrow  \frac{c_{1}c_{2} + 1}{c_{1}} > \frac{a_{1}a_{2} + 1}{a_{1}} \geq \min \{\frac{a_{1}a_{2} + 1}{a_{1}}, \frac{\alpha}{\beta} \};
\end{align*}

\item The link $L = L(a_{1}a_{2}a_{3}, R, -n)$ with $a_{3} < n$ has an equivalent description in standard form $M(1;(a_{1}a_{2}a_{3} + a_{1} + a_{3}, a_{1}a_{2} + 1), (\alpha,\beta),(n,n-1))$. So we have
\begin{align*}
n > a_{3} & \Rightarrow n > a_{3} + \frac{a_{1}}{a_{1}a_{2} + 1} \\
& \Rightarrow  n >  \frac{a_{1}a_{2}a_{3} + a_{1} + a_{3}}{a_{1}a_{2} + 1}  \geq \min \{ \frac{a_{1}a_{2}a_{3} + a_{1} + a_{3}}{a_{1}a_{2} + 1}, \frac{\alpha}{\beta} \};
\end{align*}
\item The link $L = L(a_{1}a_{2}a_{3}, R, (-c_{1})(-c_{2})(-c_{3}))$ with $ \frac{1+a_{1}a_{2}}{1 +c_{1}c_{2}} > \frac{a_{1} + a_{3} + a_{1}a_{2}a_{3}}{c_{1}+c_{3} + c_{1}c_{2}c_{3}}$ has an equivalent description in standard form $M(1;(a_{1}a_{2}a_{3} + a_{1} + a_{3}, a_{1}a_{2} + 1), (\alpha,\beta),(c_{1}c_{2}c_{3} + c_{1} + c_{3}, (c_{1}c_{2}c_{3} + c_{1} + c_{3}) - ( c_{1}c_{2} + 1)))$. So we have
\begin{align*}
\frac{1+a_{1}a_{2}}{1 +c_{1}c_{2}} > \frac{a_{1} + a_{3} + a_{1}a_{2}a_{3}}{c_{1}+c_{3} + c_{1}c_{2}c_{3}} & \iff \frac{c_{1}+c_{3} + c_{1}c_{2}c_{3}}{1 +c_{1}c_{2}} > \frac{a_{1} + a_{3} + a_{1}a_{2}a_{3}}{1+a_{1}a_{2}}\\
&\Rightarrow \frac{c_{1}+c_{3} + c_{1}c_{2}c_{3}}{1 +c_{1}c_{2}} > \frac{a_{1} + a_{3} + a_{1}a_{2}a_{3}}{1+a_{1}a_{2}} \geq \min \{ \frac{a_{1} + a_{3} + a_{1}a_{2}a_{3}}{1+a_{1}a_{2}}, \frac{\alpha}{\beta} \}.
\end{align*}
\end{enumerate}
\end{proof}

The following proposition generalizes \cite[Proposition.\,3.1]{CK}:
\begin{prop}
The Montesinos link $ L = M(e;(\alpha_{1},\beta_{1}),(\alpha_{2},\beta_{2}), \ldots,(\alpha_{r},\beta_{r}) )$ in standard form is not
 quasi-alternating if $ 1 < e < r $.

\end{prop}
\begin{proof}
The Montesinos link in this case has an equivalent description $M(0;(\alpha_{1},\beta_{1} - \alpha_{1}),(\alpha_{2},\beta_{2} -\alpha_{2}), \ldots, (\alpha_{|e|},\beta_{|e|} -\alpha_{|e|}),(\alpha_{|e|+1},\beta_{|e|+1}), \ldots, (\alpha_{r},\beta_{r}) )$ that is a  non-alternating diagram obtained by summing two strongly alternating 4-end tangles. Hence, it is adequate nonalternating diagram by \cite[Page.\,532]{LT} and by \cite[Proposition.\,5.1]{K} the link is not quasi-alternating since it is homologically thick.

\end{proof}

Finally, we give some examples of Montesinos knots that are known to be not quasi-alternating.

\begin{ex}
The knots $9_{46}$ and $11n_{50}$ are Montesinos knots that are given by the descriptions in standard form $M(1;(3,1),(3,1), (3,2))$ and $M(1;(5,3),(5,2), (3,1))$ respectively. These two knots do not satisfy the condition mentioned in the third case of theorem \ref{main} and they are known to be not quasi-alternating the first knot by \cite[Proposition.\,2.2]{G} since it is the pretzel knot $P(3,3,-3)$ and the second knot by \cite[Theorem.\,1.3]{G}.
\end{ex}

\begin{ex}
The following list of Montesinos knots in standard form are candidates for homologically thin non quasi-alternating knots as stated in \cite[Page.\,4]{JS}. We checked by hand that these knots do not satisfy the condition mentioned in the third case of theorem \ref{main}.
\begin{enumerate}
\item $12n_{145} = M(1;(5,3),(5,2),(4,1))$;
\item $13n_{1408} = M(1;(7,4),(5,2),(5,2))$;
\item $13n_{2006} = M(1;(7,4),(7,3),(3,1))$;
\item $13n_{3142} = M(1;(5,3),(11,4),(3,1))$.
\end{enumerate}

\end{ex}

We end this paper with the following conjecture that is supported by the above two examples.
\begin{conj}
Theorem \ref{main} characterizes all quasi-alternating Montesinos links in standard form. In other words, any quasi-alternating Montesinos link has to satisfy one of the conditions of the above theorem.
\end{conj}

\section{Acknowledgment}
The first author thanks the Max Planck Institute for Mathematics for the kind hospitality during the course of this work.

\end{document}